\documentclass[11pt]{article}
\usepackage{amsmath}
\usepackage{amsthm}
\usepackage{amstext}
\usepackage{amssymb}

\numberwithin{equation}{section}
\numberwithin{figure}{section}
\theoremstyle{plain}
\newtheorem*{problem*}{Problem}
\newtheorem{thm}{Theorem}[section]
  \theoremstyle{plain}
  \newtheorem{conjecture}[thm]{Conjecture}
  \theoremstyle{plain}
  \newtheorem{lem}[thm]{Lemma}
  \theoremstyle{definition}
  \newtheorem{prop}[thm]{Proposition}
  \theoremstyle{plain}
  \newtheorem{defn}[thm]{Definition}
  \theoremstyle{plain}
  \newtheorem*{thm*}{Theorem}

\title{Judicious partitions of directed graphs}
\author{
Choongbum Lee
\thanks{
Department of Mathematics, MIT, Cambridge, MA 02139. E-mail: cb\_lee@math.mit.edu.
}
\and
Po-Shen Loh
\thanks{
Department of Mathematical Sciences, Carnegie Mellon University, Pittsburgh, PA 15213. E-mail: ploh@cmu.edu.
Research supported by NSF grant DMS-1201380, an NSA Young Investigators Grant and a USA-Israel BSF Grant.}
\and
Benny Sudakov
\thanks{
Department of Mathematics, UCLA, Los Angeles, CA 90095. Email: bsudakov@math.ucla.edu. Research supported
in part by NSF grant DMS-1101185, by AFOSR MURI grant FA9550-10-1-0569 and by a USA-Israel BSF grant.}
}

\date{}

\newcommand{\rt}[1]{\overrightarrow{#1}}

\begin{document}

\global\long\def\E#1{\mathbb{E}[#1]}
\global\long\def\Var#1{\textrm{Var}\left[#1\right]}
\global\long\def\Cov#1{\textrm{Cov}[#1]}
\global\long\def\P#1{\mathbf{P}(#1)}
\global\long\def\pr#1{\mathbf{P}\left(#1\right)}
\global\long\def\MAXCUT{\textrm{Max Cut}}

\maketitle

\begin{abstract}
  The area of \emph{judicious partitioning}\/ considers the general family
  of partitioning problems in which one seeks to optimize several
  parameters simultaneously, and these problems have been widely studied in
  various combinatorial contexts.  In this paper, we study essentially the
  most fundamental judicious partitioning problem for directed graphs,
  which naturally extends the classical Max Cut problem to this setting: we
  seek bipartitions in which many edges cross in each direction.  It is
  easy to see that a minimum outdegree condition is required in order for
  the problem to be nontrivial, and we prove that every directed graph with
  $m$ edges and minimum outdegree at least two admits a bipartition in
  which at least $(\frac{1}{6}+o(1))m$ edges cross in each direction.  We
  also prove that if the minimum outdegree is at least three, then the
  constant can be increased to $\frac{1}{5}$. If the minimum outdegree
  tends to infinity with $n$, then the constant increases to $\frac{1}{4}$.
  All of these constants are best-possible, and provide asymptotic answers
  to a question of Alex Scott.
\end{abstract}

\maketitle

\section{Introduction\label{sec:Introduction}}

Partitioning problems have a long history in mathematics and theoretical
computer science. One famous example is Max Cut, which seeks a bipartition
of a given graph which maximizes the number of edges which cross between
the two sides.  This is a fundamental problem, and has been the subject of
much investigation (see, e.g., \cite{FrJe97,GoWi95,Hastad01,TSSW00} and
their references). Computing the exact solution can be quite difficult,
since the Max Cut problem is known to be NP-complete.  Still, it is
possible to obtain some estimates on the size of the Max Cut in terms of
the number of edges of the graph. A folklore bound (which comes from a
simple and efficient algorithm) asserts that every graph with $m$ edges
has 

$$
\MAXCUT\ge\frac{m}{2}.
$$
This immediately gives a 0.5-approximation algorithm, because no cut can
have size greater than the total number of edges $m$. The current best
known approximation ratio of 0.87856 is given by the celebrated algorithm of
Goemans and Williamson \cite{GoWi95}, which is based on an ingenious
application of semi-definite programming.  From a purely combinatorial
perspective, it is of interest to determine best-possible bounds for
parameters of optimal partitions.  Edwards \cite{Edwards73} improved on the
folklore bound and proved that
\[
  \MAXCUT
  \ge
  \left\lceil 
  \frac{m}{2}+\sqrt{\frac{m}{8}+\frac{1}{64}}-\frac{1}{8}
  \right\rceil,
\]
which is tight, e.g., for complete graphs.

Empowered by the growth of probabilistic techniques, a new class of
\emph{judicious}\/ partitioning results has emerged. In these problems, one
simultaneously optimizes several properties, in contrast to the classical
problems such as Max Cut where one attempts to optimize a single parameter.
A classic result in this area is a theorem of Bollob\'as and Scott
\cite{BoSc99} which asserts that every $m$-edge graph has a bipartition
$V=V_{1}\cup V_{2}$ of its vertex set in which 
\[
  e(V_{1},V_{2})
  \ge
  \left\lceil 
  \frac{m}{2}+\sqrt{\frac{m}{8}+\frac{1}{64}}-\frac{1}{8}
  \right\rceil
\]
and 
\[
\max\{e(V_{1}),e(V_{2})\}\le\frac{m}{4}+\sqrt{\frac{m}{32}+\frac{1}{256}}-\frac{1}{16}.
\]
Note that their result simultaneously optimizes three parameters: the
number of edges across the partition (matching the Edwards bound), and the
number of edges inside each $V_{i}$. We direct the interested reader to any
of \cite{ABKS, AKS, BoReTh93,BoSc00,BoSc10,Ha11,KuOs07,LeLoSu,MaYaYu10,MaYu11}
(by no means a comprehensive list), or to either of the surveys
\cite{BoSc02a,Scott06} for more background on the judicious partitioning
literature.

In this paper we study essentially the most fundamental judicious
partitioning problem for directed graphs.  A directed graph is a pair $(V,
E)$ where $V$ is a set of vertices, and $E$ is a set of distinct edges
$\rt{uv}$, where $u \neq v$.  We disallow loops and multiple edges, but do
allow both $\rt{uv}$ and $\rt{vu}$ to be present.  In this context, any cut
$V = V_1 \cup V_2$ is most naturally associated with two parameters: the
number of edges from $V_1$ to $V_2$, and the number of edges from $V_2$ to
$V_1$.  Thus, in contrast to undirected graphs, where Max Cut only
needs to optimize a single parameter (the total number of crossing edges),
in directed graphs one can measure the size of a cut in each direction.
Therefore, in the judicious analogue of the Max Cut problem for directed
graphs, one seeks a bipartition which has many edges crossing in both
directions.

Although it is easy to guarantee a partition with at least $1/4$ of the
edges in a single direction,
one immediately notices that the problem as stated above has the following issue. 
If the digraph is a star with all edges oriented
from a central vertex, then regardless of the bipartition,
one direction would always have zero edges.  This is similar to the issue which
arose in the judicious bisection problem in graphs (see \cite{LeLoSu}), and
in both cases, it can be resolved by imposing a minimum-degree condition.
The following natural question appears in the survey of Scott \cite{Scott06}. 

\begin{problem*}
  Let $d$ be a positive integer. What is the maximum constant $c_{d}$ such
  that every $m$-edge directed graph of minimum outdegree at least $d$
  admits a bipartition $V=V_{1}\cup V_{2}$ of its vertex set in which 
  \[
    \min\{e(V_{1},V_{2}),e(V_{2},V_{1})\}\ge c_{d}\cdot m?
    \]
\end{problem*}

For $d=1$, consider the graph $K_{1,n-1}$ and add a single edge inside the
part of size $n-1$. This graph can be oriented so that the minimum
outdegree is 1 and $\min\{e(V_{1},V_{2}),e(V_{2},V_{1})\}\le1$ for every
partition $V=V_{1}\cup V_{2}$. This is because we have a cyclically
oriented triangle, with lots of edges all pointing in to one of the
vertices of that triangle.  Then all of those edges will only contribute to
$e(V_1, V_2)$ or $e(V_2, V_1)$, depending on whether the apex is in $V_2$
or $V_1$, respectively.  Altogether, the other edges will only contribute a
total of at most one edge back in the other direction.  Hence we see that
$c_{1}=0$.

For $d\ge2$, first orient the edges of the complete graph $K_{2d-1}$ along
an Eulerian circuit.  In this way we obtain a directed graph with $2d-1$
vertices, and all outdegrees equal to $d-1$.  Moreover, in every
bipartition of its vertex set, the number of edges crossing in each
direction is exactly the same (this is easily seen by following the
Eulerian circuit).  Hence in every bipartition of its vertex set, the
maximum number of edges in any direction is at most $\frac{d(d-1)}{2}$. Now
consider the directed graph where we take $k$ vertex disjoint copies of
$K_{2d-1}$ oriented as above, and a single vertex disjoint copy of
$K_{2d+1}$ oriented in a similar manner. Fix a vertex $v_{0}$ of
$K_{2d+1}$, and add edges so that all the vertices belonging to the copies
of $K_{2d-1}$ are in-neighbors of $v_{0}$. This graph has minimum
outdegree $d$, and its number of edges is 
\begin{align*}
m &= k(d-1)(2d-1)+d(2d+1)+k(2d-1)  \\
  &= kd(2d-1) + d(2d+1).
\end{align*}
Moreover, for every partition $V=V_{1}\cup V_{2}$ of its vertex set
with $v_{0}\in V_{1}$, we have 
\[
e(V_{1,}V_{2})\le k\frac{d(d-1)}{2}+\frac{d(d+1)}{2}=\frac{d-1}{2(2d-1)}m+\frac{d^{2}}{2d-1}.
\]
Hence this graph shows that $c_{d}\le\frac{d-1}{2(2d-1)}$. Our main
theorem asserts that for $d=2,3$ this bound is asymptotically best
possible.
\begin{thm}
\label{thm:main}For $d=2,3$, every directed graph of minimum outdegree
at least $d$ admits a bipartition $V=V_{1}\cup V_{2}$ of its
vertex set for which 
\[
\min\{e(V_{1},V_{2}),e(V_{2},V_{1})\}\ge\left(\frac{d-1}{2(2d-1)}+o(1)\right)m.
\]
Thus $c_{2}=\frac{1}{6} + o(1)$ and $c_{3}=\frac{1}{5} + o(1)$.
\end{thm}
Based on the constructions above and Theorem \ref{thm:main} we make
the following conjecture.
\begin{conjecture}
  \label{conj:main}
  Let $d$ be an integer satisfying $d\ge4$. Every directed graph of minimum
  outdegree at least $d$ admits a bipartition $V=V_{1}\cup V_{2}$ of its
  vertex set for which 
  \[
    \min\{e(V_{1},V_{2}),e(V_{2},V_{1})\}
    \ge
    \left(\frac{d-1}{2(2d-1)}+o(1)\right)m.
    \]
\end{conjecture}
This paper is organized as follows. In Sections \ref{sec:second_moment} and
\ref{sec:bisection}, we prove the core results which drive the proof of our
main theorem.  We prove the $d=2$ case of the main theorem in Section
\ref{sec:min-2}, and the $d=3$ case in Section \ref{sec:min-3}.  The final
section contains some concluding remarks, with a discussion of the
obstacles that remain in the cases $d\ge4$.

\medskip

\noindent \textbf{Notation.} Graphs $G=(V,E)$ and directed graphs
$D=(V,E)$ are given by pairs of vertex sets and edge sets.  All of our
objects will have no loops (endpoints of edges are distinct), and no
multiple edges (edges are all distinct), although directed graphs are
permitted to have antiparallel pairs $\rt{uv}$, $\rt{vu}$.  A directed
graph is connected if the underlying undirected graph is connected. For an
undirected graph $G=(V,E)$ and two vertex subsets $X$ and $Y$, we let
$e(X,Y)=|\{xy\,:\, x\in X,y\in Y,\,xy\in E\}|$.  For a directed graph
$D=(V,E)$ and a vertex $v$, let $d^{-}(v)$ and $d^{+}(v)$ be the number of
$v$'s in-neighbors and out neighbors, respectively, and let
$d(v)=d^{-}(v)+d^{+}(v)$ be the total degree of $v$. Note that $d(v)$ can
potentially be as high as $2(n-1)$ because edges in both directions are
permitted between each pair. 
For two vertex subsets $X$ and $Y$ in a digraph,
let $e(X,Y)=|\{\rt{xy}\,:\, x\in X,y\in Y,\rt{xy}\in E\}|$. 
Let $e(X) = e(X, X)$. For a vertex
set $A$, we let $D[A]$ denote the induced subgraph of $D$ on $A$.  Since
the majority of our results are asymptotic in nature, we will implicitly
ignore rounding effects whenever these effects are of smaller order than
our error terms. For two functions $f(n)$ and $g(n)$, we write
$f(n)=o(g(n))$ if $\lim_{n\rightarrow\infty}f(n)/g(n)=0$.  We often use
subscripts such as $\varepsilon_{3.1}$ to indicate that $\varepsilon$ is
the constant coming from Theorem/Corollary/Lemma 3.1.

\section{Basic probabilistic approach}
\label{sec:second_moment}

A simple, yet powerful, method of obtaining an effective partition is to
apply randomness, by independently placing each vertex to each side with
some specified probability. Even though this method is not powerful enough
to immediately solve our main problem, it serves as a useful starting
point, and in fact provides a sufficiently good partition for some range of
the parameter space.  In this section, we develop this idea in a slightly
more general form, keeping in mind later applications.  The following lemma
estimates the number of edges across a random partition using the first and
second moment methods.

\begin{lem}
  \label{lem:secondmoment}
  Let $D=(V,E)$ be a directed graph with $m$ edges. Let $0 \leq p \leq 1$
  be a real number. Suppose that we are also given a subset $A\subset V$,
  with partition $A=A_{1}\cup A_{2}$. Let $B=V\setminus A$ and consider a
  random bipartition $B=B_{1}\cup B_{2}$ obtained by independently placing
  each vertex of $B$ in $B_{1}$ with probability $p$, and in $B_{2}$ with
  probability $1-p$. Let $V_{1}=A_{1}\cup B_{1}$ and $V_{2}=A_{2}\cup
  B_{2}$.  Then
  \begin{align*}
    \E{e(V_{1},V_{2})} & =e(A_{1},A_{2})+(1-p)\cdot e(A_{1}, B)+p\cdot e(B, A_{2})+p(1-p)\cdot e(B)\quad\textrm{and}\\
    \Var{e(V_{1},V_{2})} & <2m\cdot\max_{v\in B}d(v).
  \end{align*}
\end{lem}

\begin{proof}
For each edge $e=\rt{vw}$ of the directed graph $D$, let $\mathbf{1}_{e}$
be the indicator random variable of the event that the edge $e$ becomes
an edge from $V_{1}$ to $V_{2}$. We have 
\[
e(V_{1},V_{2})=\sum_{e}\mathbf{1}_{e}.
\]
Note that
\[
  \E{\mathbf{1}_{e}}=\left\{ \begin{split}1 & \quad \text{ if } v\in
    A_{1},w\in A_{2},\\
    1-p & \quad \text{ if } v\in A_{1},w\in B,\\
    p & \quad \text{ if } v\in B,w\in A_{2},\\
    p(1-p) & \quad \text{ if } v\in B,w\in B,\\
    0 & \quad\text{ otherwise.}
\end{split}
\right.
\]
The claim on the expected value of $e(V_{1},V_{2})$ immediately follows
from linearity of expectation.

To estimate the variance of $e(V_{1},V_{2})$, it suffices to focus
on the edges $e=\rt{vw}$ for which $(v\in A_{1},w\in B)$, $(v\in B,w\in A_{2})$,
or $(v,w\in B)$, as all other edges have constant contribution towards
$e(V_{1},V_{2})$. Let $E_{1,2}$ be the set of such edges. We have
\[
\Var{\sum_{e}\mathbf{1}_{e}}=\Var{\sum_{e\in E_{1,2}}\mathbf{1}_{e}}=\sum_{e\in E_{1,2}}\Var{\mathbf{1}_{e}}+\sum_{e,e'\in E_{1,2},e\neq e'}\Cov{\mathbf{1}_{e},\mathbf{1}_{e'}}.
\]
For $e\in E_{1,2}$, we have
$\Var{\mathbf{1}_{e}}\le\E{\mathbf{1}_{e}}\le1$.  For the second term, we
have $\Cov{\mathbf{1}_{e},\mathbf{1}_{e'}}=0$ if $e$ and $e'$ do not share
a vertex. If $e, e' \in E_{1,2}$ go between the same pair of endpoints, but
in opposite directions, then they can never simultaneously contribute to
$e(V_1, V_2)$, and hence $\Cov{\mathbf{1}_{e},\mathbf{1}_{e'}} \leq 0$.
Furthermore, if $e,e'\in E_{1,2}$ share an endpoint in $A$ but have
distinct endpoints in $B$, then 
\[
\Cov{\mathbf{1}_{e},\mathbf{1}_{e'}}=\E{\mathbf{1}_{e}\mathbf{1}_{e'}}-\E{\mathbf{1}_{e}}\E{\mathbf{1}_{e'}}=0.
\]
Hence the only positive contributions to
$\Cov{\mathbf{1}_{e},\mathbf{1}_{e'}}$ come when $e$ and $e'$ share a
vertex in $B$. Since
$\Cov{\mathbf{1}_{e},\mathbf{1}_{e'}}\le\E{\mathbf{1}_{e}\mathbf{1}_{e'}}\le1$,
we have 
\begin{align*}
\sum_{e,e'\in E_{1,2},e\neq e'}\Cov{\mathbf{1}_{e},\mathbf{1}_{e'}} & \le\sum_{v\in B}d(v)(d(v)-1)\\
 & \le\left(\sum_{v\in B}d(v)\right) \left(\max_{v\in B}d(v)-1\right)\\
 & \le2m\left(\max_{v\in B}d(v)-1\right).
\end{align*}
Thus 
\[
\Var{\sum_{e\in E_{1,2}}\mathbf{1}_{e}}\le\left(\sum_{e\in
  E_{1,2}}1\right)+2m\left(\max_{v\in B}d(v)-1\right)<2m\cdot\max_{v\in B}d(v).
\]

\end{proof}
This implies the following lemma.
\begin{lem}
\label{lem:partition_secondmoment}Let $D=(V,E)$ be a given directed
graph with $m$ edges. Let $p$ be a real satisfying $p\in[0,1]$,
and $\varepsilon$ be a positive real. Suppose that a subset $A\subset V$
and its partition $A=A_{1}\cup A_{2}$ are given, and let $B=V\setminus A$.
Further suppose that $\max_{v\in B}d(v)\le\frac{\varepsilon^{2}}{4}m$.
Then there exists a partition $V_{1}\cup V_{2}$ for which 
\begin{align*}
e(V_{1},V_{2}) & \ge e(A_{1},A_{2})+(1-p)\cdot e(A_{1}, B)+p\cdot e(B, A_{2})+p(1-p)\cdot e(B)-\varepsilon m\quad\textrm{and}\\
e(V_{2},V_{1}) & \ge e(A_{2},A_{1})+p\cdot e(A_{2}, B)+(1-p)\cdot e(B, A_{1})+p(1-p)\cdot e(B)-\varepsilon m.
\end{align*}
\end{lem}
\begin{proof}
Let $V_{1}\cup V_{2}$ be the partition obtained by placing each vertex
in $B$ independently in $V_{1}$ or $V_{2}$, with probability $p$
and $1-p$, respectively. Let
\[
m_{1,2}=e(A_{1},A_{2})+(1-p)\cdot e(A_{1}, B)+p\cdot e(B, A_{2})+p(1-p)\cdot e(B),
\]
and 
\[
m_{2,1}=e(A_{2},A_{1})+p\cdot e(A_{2}, B)+(1-p)\cdot e(B, A_{1})+p(1-p)\cdot e(B).
\]
By Lemma \ref{lem:secondmoment} and Chebyshev's inequality,
\[
\mathbf{P}\Big(e(V_{1},V_{2})\ge m_{1,2}-\varepsilon m\Big)\le\frac{\Var{e(V_{1},V_{2})}}{\varepsilon^{2}m^{2}}<\frac{2m\cdot\max{}_{v\in B}d(v)}{\varepsilon^{2}m^{2}}\le\frac{(\varepsilon^{2}/2)m^{2}}{\varepsilon^{2}m^{2}}=\frac{1}{2}.
\]
Similarly, we have 
\[
\mathbf{P}\Big(e(V_{1},V_{2})\ge m_{2,1}-\varepsilon m\Big)<\frac{1}{2}.
\]
Hence there exists a partition $V=V_{1}\cup V_{2}$ for which $e(V_{1},V_{2})\ge m_{1,2}-\varepsilon m$
and $e(V_{2},V_{1})\ge m_{2,1}-\varepsilon m$ both hold.
\end{proof}

The next statement is an immediate corollary of the lemma.

\begin{prop} 
  \label{prop:second_moment}
  For $\varepsilon > 0$, let $D=(V,E)$ be an $n$-vertex directed graph with
  $m$ edges, such that all degrees are at most
  $\frac{\varepsilon^{2}}{4}m$, or $m \geq 8 \varepsilon^{-2}n$.
  Then there exists a partition $V_{1}\cup V_{2}$ for which both
  $e(V_{1},V_{2})$ and $e(V_{2},V_{1})$ are at least
  $\big(\frac{1}{4}-\varepsilon\big)m$.
\end{prop}

Note that $m\ge 8\varepsilon^{-2}n$ implies that the maximum degree is at most $\frac{\varepsilon^2}{4}m$,
since all degrees of a directed graph are at most $2n$. Hence Proposition \ref{prop:second_moment} indeed is an immediate corollary of the lemma
(where we take $p=\frac{1}{2}$ and $A_1 = A_2 = \emptyset$).

\section{Large bipartition\label{sec:bisection}}

As noticed in \cite{ErGyKo97,PoTu82}, the results that Edwards proved
in \cite{Edwards75} implicitly imply that connected graphs with $n$
vertices and $m$ edges admit a bipartition of size at least
\[
\frac{m}{2}+\frac{n-1}{4}.
\]
In fact, for even integers $n$ we have $\lceil\frac{m}{2}+\frac{n-1}{4}\rceil\ge\frac{m}{2}+\frac{n}{4}$,
and thus the above bound implies that a graph with $\tau$ odd components
admits a bipartition of size at least 
\[
\frac{m}{2}+\frac{n-\tau}{4}.
\]

A bisection of a graph is a bipartition of its vertex set in which the
number of vertices in the two parts differ by at most one. In
\cite{LeLoSu}, we extended the bound above to bisections and proved that
every graph with $n$ vertices, $m$ edges, $\tau$ odd components, and
maximum degree $\Delta$ admits a bisection of size at least
\[
\frac{m}{2}+\frac{n-\max\{\tau,\Delta-1\}}{4}.
\]

We then developed a randomized algorithm which asymptotically achieves
the bound above (and some other estimates as well), based on the proof
of this theorem. This algorithm turned out to be a powerful new tool
in obtaining a judicious bisection result. In this paper, we adjust
the randomized algorithm for directed graphs. The following theorem
is one of the main tools of this paper.
\begin{thm}
\label{thm:random_bisection}Given any real constants $C,\varepsilon>0$,
there exist $\gamma,n_{0}>0$ for which the following holds. Let $D=(V,E)$
be a given directed graph with $n\geq n_{0}$ vertices and at most
$Cn$ edges, and let $A\subset V$ be a set of at most $\gamma n$
vertices which have already been partitioned into $A_{1}\cup A_{2}$.
Let $B=V\setminus A$, and suppose that every vertex in $B$ has degree
at most $\gamma n$ (with respect to the full $D$). Let $\tau$ be
the number of odd components in $D[B]$. Then, there is a bipartition
$V=V_{1}\cup V_{2}$ with $A_{1}\subset V_{1}$ and $A_{2}\subset V_{2}$,
such that both 
\begin{align*}
e(V_{1},V_{2}) & \ge e(A_{1},A_{2})+\frac{e(A_{1},B)+e(B,A_{2})}{2}+\frac{e(B)}{4}+\frac{n-\tau}{8}-\varepsilon n\\
e(V_{2},V_{1}) & \ge e(A_{2},A_{1})+\frac{e(B,A_{1})+e(A_{2},B)}{2}+\frac{e(B)}{4}+\frac{n-\tau}{8}-\varepsilon n\,.
\end{align*}

\end{thm}
Informally, Theorem \ref{thm:random_bisection} asserts that if the number
of edges and the maximum degree satisfy certain conditions, then we can in
fact obtain an additive term of $\frac{n-\tau}{8}$ over the expected number
of edges in a purely random bipartition.  Consider the directed graphs
given in Section \ref{sec:Introduction} which achieve the upper bound of
Theorem \ref{thm:main}. In the notation of Theorem
\ref{thm:random_bisection}, $A$ is the set whose only element is the vertex
of degree $n-1$, and $B$ is the set of other vertices. Note that the
induced subgraph on $B$ consists of components of odd size. These graphs
are designed to maximize $\tau$, and hence these graphs will turn out to be
the graphs which give the worst bound in Theorem
\ref{thm:random_bisection}.  The proof of this theorem is somewhat
involved, although it is similar to the that of the corresponding theorem
in \cite{LeLoSu}.  The rest of this section is devoted to its proof.

\subsection{Decomposing the graph}

We start with a technical lemma which will provide structural information
about the underlying undirected (simple) graph obtained by ignoring edge
orientations and removing redundant parallel edges when edges in both
directions appear between pairs of vertices.  A \emph{star}\/ is a
bipartite graph on $n$ vertices consisting of a unique vertex of degree
$n-1$, and $n-1$ other vertices of degree one. We refer to the unique
vertex of degree $n-1$ as the \emph{apex}. The following lemma decomposes
an undirected graph (with no loops or multiple edges) into induced stars
plus some leftover vertices.

\begin{lem}
  \label{lem:star_decompose}
  Let $\varepsilon$ and $C$ be arbitrary positive reals.  Let $G$ be an
  undirected graph with $n$ vertices, $m \leq Cn$ edges, maximum degree
  $\Delta$, and $\tau$ odd components.  Then there exists a partition
  $V=T_{1}\cup T_{2}\cup\cdots\cup T_{s}\cup U$ of its vertex set such that 
  \begin{description}
    \item[(i)] each $T_{i}$ induces a star, and $2\le|T_{i}|\le\Delta+1$,
    \item[(ii)] all but at most one non-apex vertex in each $T_{i}$ has
      degree (in the full graph) at most $\frac{2C}{\varepsilon}$, and
    \item[(iii)] $U$ is an independent set of order $|U|\le\tau+\varepsilon n$.
  \end{description}
\end{lem}
The lemma above is implicitly proved in \cite{LeLoSu}. A similar lemma also
appears in the paper of Erd\H{o}s, Gy\'arfas, and Kohayakawa
\cite{ErGyKo97}, but their bound is in terms of the number of connected
components, not the number of odd components.  In order to prove this
lemma, we first take a maximum matching. Afterwards, for the leftover
vertices which are not covered by the matching, we attempt to find an edge
in the matching with which the vertex will create an induced star. By
systematically assigning each leftover vertex in this way, we will
eventually obtain the partition described in Lemma
\ref{lem:star_decompose}.  In order to provide the full details for this
argument, it is convenient to introduce the following concept.

\medskip

\begin{defn}
Let $\{e_{1},\ldots,e_{s}\}$ be the edges of a maximum matching in
a graph $G=(V,E)$, and let $W$ be the set of vertices not in the
matching. With respect to this fixed matching, say that a vertex $v$
in a matching edge $e_{i}$ is a \emph{free neighbor} of a vertex
$w\in W$ if $w$ is adjacent to $v$, but $w$ is not adjacent to
the other endpoint of $e_{i}$. In this case, we also say that $e_{i}$
is a \emph{free neighbor} of $w$. Call a vertex $w\in W$ a \emph{free
vertex} if it has at least one free neighbor.
\end{defn}
A \emph{tight component} is a connected component $T$ such that for every
$v\in T$, the subgraph induced by $T\setminus\{v\}$ contains a perfect
matching, and every perfect matching of $T\setminus\{v\}$ has the property
that no edge of the perfect matching has exactly one endpoint adjacent to
$v$. Note that a tight component is necessarily an odd component.  The
following lemma delineates the relationship between non-free vertices and
tight components.
\begin{lem}
\label{lem:freeandtight} Let $\{e_{1},\ldots,e_{s}\}$ be the edges
of a maximum matching in an undirected graph $G=(V,E)$, and let $W$ be the
set of vertices not in the matching. Further assume that among all
matchings of maximum size, we have chosen one which maximizes the number of
free vertices in $W$. Then, every tight component contains a distinct
non-free vertex of $W$, and all non-free $W$-vertices are covered in this
way (there is a bijective correspondence).\end{lem}
\begin{proof}
\noindent The matching $\{e_{1},\ldots,e_{s}\}$ must be maximal within
each connected component. One basic property of a tight component
is that it contains an almost-perfect matching which misses only one
vertex. Consequently, by maximality, in every tight component $T$,
$\{e_{1},\ldots,e_{s}\}$ must miss exactly one vertex $w\in W$.
Furthermore, the second property of a tight component is that $w$
must have either 0 or 2 neighbors in each edge $e_{i}$ in $T$ (and
$w$ must have 0 neighbors in each edge $e_{j}$ not in $T$, since
$T$ is the connected component containing $w$). Therefore, the unique
vertex $w$ is in fact a non-free $W$-vertex contained in $T$.

The remainder of the proof concentrates on the more substantial part
of the claim, which is that each non-free $W$-vertex is contained
in some tight component. Consider such a vertex $w$, and let $T$
be a maximal set of vertices which (i) contains $w$, (ii) induces
a connected graph which is a tight component, and (iii) does not cut
any $e_{i}$. Since the set $\{w\}$ satisfies (i)--(iii), our optimum
is taken over a non-empty set, and so $T$ exists.

If $T$ is already disconnected from the rest of the graph, then we
are done. So, consider a vertex $v\not\in T$ which has a neighbor
$v'\in T$. If $v\in W$, then we can modify our matching by taking
the edge $vv'$, and changing the matching within $T$ by using property
(ii) to generate a new matching of $T\setminus\{v'\}$. This will
not affect the matching outside of $T\cup\{v\}$, because property
(iii) insulates the adjustments within $T$ from the rest of the matching
outside. We would then obtain a matching with one more edge, contradicting
maximality. Therefore, all vertices $v\not\in T$ which have neighbors
in $T$ also satisfy $v\not\in W$.

Let us then consider a vertex $v_{1}\not\in T \cup W$ with a neighbor $v'\in T$.
We now know that $v_{1}$ must be covered by a matching edge; let
$v_{2}$ be the other endpoint of that edge. By (iii), we also have
$v_{2}\not\in T$. Note that $v_{2}$ cannot have a neighbor $w'\in W\setminus T$,
or else we could improve our matching by replacing $v_{1}v_{2}$ with
$w'v_{2}$ and $v_{1}v'$, and then using (ii) to take a perfect matching
of $T\setminus\{v'\}$.

Our next claim is that $v_{2}$ must be adjacent to $v'$ as well.
Indeed, assume for contradiction that this is not the case. Then,
consider modifying our matching by replacing $v_{1}v_{2}$ with the
edge $v_{1}v'$ and changing the matching within $T$ by using (ii)
to generate a new matching of $T\setminus\{v'\}$. As before, (iii)
ensures that the result is still a matching. This time, the new matching
has the same size as the original one, but with more free $W$-vertices
(note that the vertex $v_{2}$ replaced the vertex $w$ in the set
$W$). To see this, observe that $v_{2}$ is now unmatched and free
because it is adjacent to $v_{1}$ but not $v'$. Previously, the
only $W$-vertex inside $T$ was our original $w$, which we assumed
to be non-free in the first place. Also, no other vertices outside
of $T$ changed from being free to non-free, because we already showed
that no $W$-vertices outside of $T$ were adjacent to $T\cup\{v_{2}\}$,
and so any vertex that was free by virtue of its adjacency with $v_{1}$
but not $v_{2}$ is still free because it is not adjacent to $v'$
either. This contradiction to maximality establishes that $v_{2}$
must be adjacent to $v'$.

We now have $v_{1}$, $v_{2}$, and $v'$ all adjacent to each other,
and no vertices of $W\setminus T$ are adjacent to $T\cup\{v_{1},v_{2}\}$.
Our argument also shows that for any $v''\in T$ which is adjacent
to one of $v_{1}$ or $v_{2}$, it also must be adjacent to the other.
Our final objective is to show that $T'=T\cup\{v_{1},v_{2}\}$ also
satisfies (i)--(iii), which would contradict the maximality of $T$.
Properties (i) and (iii) are immediate, so it remains to verify the
conditions of a tight component. Since $T$ is tight and $v_{1}v_{2}$
is an edge, $T'\setminus\{u\}$ has a perfect matching for any $u\in T$.
The tightness of $T$ and the pairwise adjacency of $v_{1}$, $v_{2}$,
and $v'$ also produce this conclusion if $u\in\{v_{1},v_{2}\}$.
It remains to show that for any $u\in T'$ and any perfect matching
of $T'\setminus\{u\}$, $u$ has either 0 or 2 vertices in every matching
edge. But if this were not the case, then we could replace the matching
within $T'$ with the violating matching of $T'\setminus\{u\}$. The
two matchings would have the same size, but $u$ would become a free
vertex. No other vertex of $W$ is adjacent to $T'$ by our observation
above, so the number of free vertices would increase, contradicting
the maximality of our initial matching. Therefore, $T'$ induces a
tight component, contradicting the maximality of $T$. We conclude
that $T$ must have been disconnected from the rest of the graph,
as required. 
\end{proof}

We are now ready to prove Lemma \ref{lem:star_decompose}.

\begin{proof}[Proof of Lemma \ref{lem:star_decompose}]
Start by taking a maximum matching $\{e_{1},\ldots,e_{s}\}$ which
secondarily maximizes the number of free vertices in
$W=V\setminus\{e_{1},\ldots,e_{s}\}=\{w_{1},\ldots,w_{r}\}$, so that we can
apply Lemma \ref{lem:freeandtight}.  By maximality, $W$ is an independent
set.  Let $U\subset W$ be the set of vertices which are either not free, or
have degree at least $\frac{2C}{\varepsilon}$. Since all tight components
have odd order, by Lemma \ref{lem:freeandtight}, there are at most $\tau$
non-free vertices. On the other hand, since there are at most $Cn$ edges
in total, there are at most $\varepsilon n$ vertices which have degree at
least $\frac{2C}{\varepsilon}$. Hence $|U|\le\tau+\varepsilon n$, giving
(iii).

We now construct the induced stars. Let $T_{i}$ be the union of the set of
vertices of $e_{i}$ and the set of vertices $w\in W\setminus U$ for which
$i$ is the minimum index where $e_{i}$ is a free neighbor of $w$.  This is
a partition $V=T_{1}\cup T_{2}\cup\cdots\cup T_{s}\cup U$ because each
vertex in $W\setminus U$ has at least one free neighbor by construction.
Since our matching is maximal, there cannot be any $e_{i}=vv'$ such
that $vw$ and $v'w'$ are both edges to distinct vertices $w,w' \in W$. So,
if two vertices $w,w'\in W$ each have a free neighbor in an $e_{i}$, then
their free neighbor is the same vertex.  This, together with the fact that
$W$ is an independent set, implies that each $T_{i}$ induces a star, giving
(i).  Finally, all vertices in each $T_{i}$ outside of $e_{i}$ have degree
at most $\frac{2C}{\varepsilon}$, and thus (ii) holds.
\end{proof}

\subsection{Randomized algorithm}

In this subsection, we use the following martingale concentration result
(essentially the Hoeffding-Azuma inequality) to control the performance of
the randomized partitioning algorithm at the heart of Theorem
\ref{thm:random_bisection}.
\begin{thm}
\label{thm:azuma} (Corollary 2.27 in \cite{JLR}.) Given real numbers
$\lambda,C_{1},\ldots,C_{n}>0$, let $f:\{0,1\}^{n}\rightarrow\mathbb{R}$
be a function satisfying the following Lipschitz condition: whenever
two vectors $z,z'\in\{0,1\}^{n}$ differ only in the $i$-th coordinate
(for any $i$), we always have $|f(z)-f(z')|\leq C_{i}$. Suppose
$X_{1},X_{2},\ldots,X_{n}$ are independent random variables, each
taking values in $\{0,1\}$. Then, the random variable $Y=f(X_{1},\ldots,X_{n})$
satisfies 
\[
\pr{|Y-\E Y|\geq\lambda}\leq2\exp\left\{ -\frac{\lambda^{2}}{2\sum C_{i}^{2}}\right\} .
\]
\end{thm}

We are now ready to prove Theorem \ref{thm:random_bisection}, which will be
the core result for our main theorem.

\medskip

\noindent \emph{Proof of Theorem \ref{thm:random_bisection}.}
Without loss of generality, assume that $C > 1$, $\epsilon < 1$, and $n$ is
sufficiently large.  Apply Lemma \ref{lem:star_decompose} to the underlying
undirected (simple) graph induced by $B$, and obtain a partition
$B=T_{1}\cup\cdots\cup T_{s}\cup U$. Note that since we are assured that
the full digraph $D$ contains at most $Cn$ edges, we can actually establish
in part (ii) of Lemma \ref{lem:star_decompose} that the degree bound of
$\frac{2C}{\varepsilon}$ holds with respect to total degrees (in- plus
outdegrees) in the full digraph $D$, not just in the undirected simple
graph on $B$.  (When constructing $U \subset W$ in the proof of that lemma,
one may absorb all vertices with degree greater than
$\frac{2C}{\varepsilon}$ with respect to the whole graph, not just in the
underlying undirected graph on $B$.) So, we may assume that each $T_i$ has
at most one non-apex vertex with full $D$-degree greater than
$\frac{2C}{\varepsilon}$, and we still have $|U|\le\tau+\varepsilon n$.
Let $v_{i}$ denote the apex vertex of tree $T_{i}$, arbitrarily
distinguishing an apex if $T_i$ has only two vertices.  We now randomly
construct a bipartition $V=V_{1}\cup V_{2}$ by placing each $A_i$ in $V_i$,
and partitioning each $T_i$ by independently placing each apex $v_{i}$ on a
uniformly random side, and then placing the rest of
$T_{i}\setminus\{v_{i}\}$ on the other side. Each remaining vertex (from
the set $U$) is independently placed on a uniformly random side.

Define the random variables $Y_{1}=e(V_{1},V_{2})$ and
$Y_{2}=e(V_{2},V_{1})$.  For an edge $e=\rt{vw}$ of the digraph, let
$\mathbf{1}_{e}$ be the indicator random variable of the event that $v\in
V_{1}$ and $w\in V_{2}$.  Thus $Y_{1}=\sum_{e}\mathbf{1}_{e}$ and
\[
\E{Y_{1}}=\sum_{e}\E{\mathbf{1}_{e}}.
\]
We have $\E{\mathbf{1}_{e}}=1$ if $v\in A_{1}$ and $w\in A_{2}$, and
$\E{\mathbf{1}_{e}}=\frac{1}{2}$ if either $v\in A_{1}$ and $w\in B$, or
$v\in B$ and $w\in A_{2}$. For edges in $D[B]$, the gain comes from edges
$e$ in the digraph which correspond to edges in the stars $T_i$ in the
underlying undirected graph on $B$: there, we have
$\E{\mathbf{1}_{e}}=\frac{1}{2}$, while all other edges in $D[B]$ give the
regular $\E{\mathbf{1}_{e}}=\frac{1}{4}$.  Note that the total number of
edges in the stars induced by the sets $T_{i}$ is at least
$\frac{|B|-|U|}{2}\ge\frac{(n-\gamma n)-(\tau+\varepsilon n)}{2}$.
Therefore, 
\begin{equation}
  \E{Y_{1}} 
  \ge 
  e(A_{1},A_{2})+\frac{e(A_{1},B)
  +e(B,A_{2})}{2}+\frac{e(B)}{4}
  +\frac{1}{4}\cdot\frac{(n-\gamma n)-(\tau+\varepsilon n)}{2}.
  \label{ineq:random_split_gain}
\end{equation}
Similarly, we have
\[
  \E{Y_{2}}
  \ge 
  e(A_{2},A_{1})+\frac{e(B,A_{1})+e(A_{2},B)}{2}
  +\frac{e(B)}{4}+\frac{n-\tau}{8}
  -\frac{(\varepsilon + \gamma) n}{8}.
\]

For each $1\leq i\leq s$, let $C_{i}$ be the sum of the degrees of all
vertices in $T_{i}$. Clearly, flipping the assignment of $v_{i}$ cannot
affect $Y_{1}$ by more than $C_{i}$. Also, flipping the assignment of any
$w\in U$ cannot change $Y_{1}$ by more than the degree $d(w)$ of $w$.
Therefore, if we define 
\[
  L
  =
  \sum_{i=1}^{s}\left(\sum_{u\in T_{i}}d(u)\right)^{2}
  +\sum_{w\in U}d(w)^{2}\,,
\]
the Hoeffding-Azuma inequality (Theorem \ref{thm:azuma}) gives 
\begin{equation}
  \pr{Y_{1}\leq\E{Y_{1}}-\frac{\varepsilon n}{2}}\leq2\exp\left\{ -\frac{\varepsilon^{2}n^{2}}{8L}\right\} \,.\label{ineq:Y1-with-L}
\end{equation}
Let us now control $L$. Each $T_{i}$ induces a star with apex $v_{i}\in
e_{i}$.  Let $u_{i}$ be the unique non-apex vertex with degree greater than
$\frac{2C}{\varepsilon}$ (if no such vertex exists, then let $u_{i}$ be an
arbitrary non-apex vertex). Since every vertex in $T_{i}$ other than
$v_{i}$ and $u_{i}$ has degree at most $\frac{2C}{\varepsilon}$ in the full
$D$, we see that 
\[
  \sum_{u\in T_{i}}d(u)
  \leq 
  d(v_{i})+d(u_{i})+(d(v_{i})-1)\cdot\frac{2C}{\varepsilon}
  \leq
  (d(u_{i})+d(v_{i}))\frac{4C}{\varepsilon}\,,
\]
and hence
\begin{align*}
L & \leq\frac{16C^{2}}{\varepsilon^{2}}\sum_{i=1}^{s}(d(u_{i})+d(v_{i}))^{2}+\sum_{w\in U}d(w)^{2}\\
 & \leq\frac{32C^{2}}{\varepsilon^{2}}\sum_{i=1}^{s}(d(u_{i})^{2}+d(v_{i})^{2})+\sum_{w\in U}d(w)^{2}\\
 & \leq\frac{32C^{2}}{\varepsilon^{2}}\sum_{v\in B}d(v)^{2}
 \le\frac{32C^{2}}{\varepsilon^{2}}(\gamma n)\sum_{v\in B}d(v)\\
 & \leq\frac{32C^{2}}{\varepsilon^{2}}(\gamma n)(2Cn)\,,
\end{align*}
where we used that $d(v)\leq\gamma n$ for all $v\in B$, and the degree sum of $D$ is at
most $2Cn$. Therefore, we choose $\gamma=\frac{\varepsilon^{4}}{1024C^{3}}$,
so that $L\leq\frac{\varepsilon^{2}n^{2}}{16}$. Substituting this into
\eqref{ineq:Y1-with-L}, we conclude that 
\[
\pr{Y_{1}\leq\E{Y_{1}}-\frac{\varepsilon n}{2}}\leq2e^{-2}<\frac{1}{2}\,,
\]
as desired. By symmetry, we have $\P{Y_{2}\le\E{Y_{2}}-\frac{\varepsilon
n}{2}}<\frac{1}{2}$ as well.  Hence, there is a partition $V=V_{1}\cup
V_{2}$ with the properties that \emph{$Y_{1}\ge\E{Y_{1}}-\frac{\varepsilon
n}{2}$} and \emph{$Y_{2}\ge\E{Y_{2}}-\frac{\varepsilon n}{2}$}, which is a
desired partition.  \hfill $\Box$

\section{Minimum outdegree two\label{sec:min-2}}
\label{sec:2}

In this section we prove the $d=2$ case of Theorem \ref{thm:main}.  Suppose
that $D=(V,E)$ is a given directed graph of minimum outdegree at least 2
with $n$ vertices and $m$ edges, and $\varepsilon$ is a given fixed
positive real number.  Our goal is to find a partition $V=V_{1}\cup V_{2}$
for which
$\min\{e(V_{1},V_{2}),e(V_{2},V_{1})\}\ge(\frac{1}{6}-\varepsilon)m$.
Throughout the proof we tacitly assume that the number of vertices $n$ is
large enough, and since $m \ge 2n$, it follows that $m$ is also large
enough.

Suppose that $m\ge1152n$. Then 
by applying Proposition \ref{prop:second_moment} with
$\varepsilon_{\ref{prop:second_moment}}=\frac{1}{12}$, we obtain a partition
$V=V_{1}\cup V_{2}$ for which 
\[
\min\{e(V_{1,}V_{2}),e(V_{2},V_{1})\}\ge\frac{m}{4}-\frac{m}{12}=\frac{m}{6}.
\]
Hence it suffices to consider the case when $m<1152n$. Since the minimum
outdegree is at least two, we see that $2n\le m<1152n.$

Let $A$ be the set of \emph{large}\/ vertices, which are defined to be the
vertices with total degree at least $n^{3/4}$, and let $B=V\setminus A$.
Note that 
\[
|A|\cdot n^{3/4}\le2m<2304n,
\]
from which it follows that $|A|\le2304n^{1/4}\le\varepsilon n$, and
$e(A)\le2304^{2}n^{1/2}\le\frac{\varepsilon m}{2}$. For sake of simplicity
we remove all the edges within $A$, and update $m$ to be the new total
number of edges in the digraph.  In terms of this new $m$, it suffices to
obtain a partition $V=V_{1}\cup V_{2}$ for which 
\[
\min\{e(V_{1,}V_{2}),e(V_{2},V_{1})\}\ge\frac{m}{6}-\frac{\varepsilon m}{2}.
\]
Indeed, this will be a desired partition even after recovering the removed
edges, since at most $\frac{\varepsilon m}{2}$ edges were removed from
within $A$.  So, for the remainder of this section, we focus on the case
when $A$ induces no edges, and we let $m_{A}=e(A,B)$ and $m_{B}=e(B)$.
Note that now $m=m_{A}+m_{B}$.

\subsection{Partition of large vertices}
\label{sec:2;gap}

Given a partition $A=A_{1}\cup A_{2}$, define
\begin{equation}
  \Theta=(e(A_{1}, B)+e(B, A_{2})\big)-\big(e(B, A_{1})+e(A_{2}, B)\big).
  \label{eq:2;gap-def}
\end{equation}
We call $\Theta$ the \emph{gap}\/ of the partition.  Consider the following
greedy algorithm to partition $A$.  Since we now assume that $A$ induces no
edges, each vertex $v \in A$ contributes one of $\pm (d^+(v) - d^-(v))$ to
the expression in \eqref{eq:2;gap-def}, and thus each contribution is
bounded in magnitude by $n$.  Process the vertices of $A$ in an arbitrary
sequential order $v_1, v_2, \ldots$, and when assigning $v_i$ to a side,
choose the side which makes the sign of $v_i$'s contribution to
\eqref{eq:2;gap-def} opposite to the sign of the cumulative contribution of
all previous $v_1, \ldots, v_{i-1}$ thus far.  Since each contribution is
bounded in magnitude by $n$, the final gap $\Theta$ of this greedy
partition will also be bounded in magnitude by $n$.  Now, let $A_1 \cup
A_2$ be a partition of $A$ which minimizes $|\Theta|$, and without loss of
generality, assume that $\Theta \geq 0$.  The greedy partition provides the
upper bound $\Theta \leq n$.

Since $\max_{v\in B}d(v)\le n^{3/4}\le\frac{\varepsilon^2}{16}m$, 
by Lemma \ref{lem:partition_secondmoment} with $p=\frac{1}{2}$
and $\varepsilon_{\ref{lem:partition_secondmoment}}=\frac{\varepsilon}{2}$,
there exists a partition $V=V_{1}\cup V_{2}$ for which 
\begin{align*}
  \min\{e(V_{1},V_{2}),e(V_{1},V_{2})\} &
  \ge\frac{1}{2}\min\{e(A_{1},B)+e(B,A_{2}),e(B,A_{1})+e(A_{2},B)\}+\frac{1}{4}e(B)-\frac{\varepsilon}{2}m\\
 & =\frac{1}{2}\cdot\frac{m_{A}-\Theta}{2}+\frac{1}{4}m_{B}-\frac{\varepsilon}{2}m=\frac{m-\Theta}{4}-\frac{\varepsilon}{2}m.
\end{align*}
Hence if $\Theta\le\frac{m}{3}$, then we obtain a desired partition.
Thus we assume for the remainder that 
\begin{equation}
\Theta>\frac{m}{3}.\label{eq:gap_min_two}
\end{equation}

For a vertex $v\in A$, we let the \emph{in-surplus}\/ of $v$ be
$s^{-}(v)=d^{-}(v)-d^{+}(v)$, and the \emph{out-surplus}\/ of $v$ be
$s^{+}(v)=d^{+}(v)-d^{-}(v)$.  Let the \emph{surplus}\/ of $v$ be
$s(v)=\max\{s^{+}(v),s^{-}(v)\}.$  Note that the in-surplus and out-surplus
differ only in their sign, and the surplus is equal to their magnitude.
Call a vertex $v\in A$ a \emph{forward vertex}\/ if either $v\in A_{1}$ and
$s^{+}(v)>0$, or $v\in A_{2}$ and $s^{-}(v)>0$. Similarly, call a vertex
$v\in A$ a \emph{backward vertex}\/ if either $v\in A_{1}$ and
$s^{-}(v)>0$, or $v\in A_{2}$ and $s^{+}(v)>0$. Observe that $\Theta$ is
the difference between the sum of surpluses of forward vertices and the sum
of surpluses of backward vertices. Let the \emph{forward edges}\/ be the
edges out of $A_{1}$, and the edges in to $A_{2}$.  Similarly, let the
\emph{backward edges}\/ be the edges in to $A_{1}$, and the edges out of
$A_{2}$. Let $m_{A}^{f}$ and $m_{A}^{b}$ be the numbers of forward and
backward edges, respectively.

\subsection{Structure of large vertices}

Call a vertex \emph{huge}\/ if $s(v)\ge\Theta$. If there are no huge
vertices, then the greedy algorithm of the previous section will
immediately give a partition of the large vertices which has gap smaller
than $\Theta$, contradicting the minimality of $\Theta$. Hence there exists
at least one huge vertex. Suppose that the vertex $v_{0}$ of largest
surplus has surplus $\Delta$.  By our analysis of the greedy partition of
$A$, we must have $\Theta \leq \Delta$.  Yet if the sum of the surpluses of
the remaining vertices of $A$ is at least $\Delta+\Theta$, then the total
number of edges is at least 
\[
m\ge\Delta+(\Delta+\Theta)\ge3\Theta>m
\]
by \eqref{eq:gap_min_two}, and this is a contradiction. Hence the sum $g$
of the surpluses of the remaining vertices of $A$ is strictly less than
$\Delta+\Theta$.  Consider the partition of $A$ which puts $v_0$ in $A_1$,
and places all other vertices of $A$ such that their surplus contributes
oppositely to the surplus of $v_0$.  The gap of the resulting partition
would have magnitude $|\Delta - g|$, and therefore, the above observation,
together with the minimality of $\Theta$, implies that $g \leq
\Delta-\Theta$.  Yet our minimal partition achieves a gap of exactly
$\Theta$, and therefore it must have a single forward vertex of surplus
$\Delta$, and all the other large vertices of positive surpluses must be
backward vertices with surpluses summing to exactly $\Delta-\Theta$. 

Note that the edges contributing to $\sum_{v\in A}\left(d(v)-s(v)\right)$
come in pairs of in-edges and out-edges. Call these the \emph{buffer
edges}\/, and let $2b=\sum_{v\in A}(d(v)-s(v))$. The observation above
implies that $m_{A}^{f}=\Delta+b$ and $m_{A}^{b}=\Delta-\Theta+b$.
Moreover, since the graph has minimum outdegree at least two, and
there are at least $b$ buffer edges directed out of $A$, it also
implies that the total number of edges in $D$ is at least 
\begin{equation}
  m 
  \ge 
  b+2|B|
  \ge
  b + 2n -2\varepsilon n.
  \label{eq:num_edges_min_two}
\end{equation}
Note that we in fact have 
\[ m \ge b + 2|B| + \sum_{v\in A} \max\{s^{+}(v), 0\}, \]
and thus the first inequality in (\ref{eq:num_edges_min_two}) is tight only if
all vertices in $A$ have in-surplus.

\subsection{Obtaining a large partition}
\label{sec:2;finish}

By the given condition, we know that all the vertices in $B$ have at least
two out-edges incident to them. At most one out-edge of each vertex can be
incident to $v_{0}$ (the vertex of largest surplus), and there are at most
$\Delta-\Theta+b$ edges directed into $A$ which are not incident to
$v_{0}$. Therefore, the induced digraph on $B$ has at most
$\Delta-\Theta+b$ isolated vertices.  Since all the other odd components of
$D[B]$ have size at least three, this implies that the number of odd
components is at most 
\[
\tau\le(\Delta-\Theta+b)+\frac{|B|-(\Delta-\Theta+b)}{3}\le\frac{n+2(\Delta-\Theta+b)}{3}.
\]
Let $\gamma$ be the constant from Theorem \ref{thm:random_bisection} where
$C=1152$ and
$\varepsilon_{\ref{thm:random_bisection}}=\frac{\varepsilon}{4}$.  Since
$|A| \leq 2304 n^{1/4} \leq \gamma n$ and the maximum degree of vertices in
$B$ is at most $n^{3/4}\le\gamma n$, by Theorem \ref{thm:random_bisection},
we obtain a partition $V=V_{1}\cup V_{2}$ for which 
\begin{align*}
\min\{e(V_{1},V_{2}),e(V_{2},V_{1})\} & \ge\frac{m_{A}-\Theta}{4}+\frac{m_{B}}{4}+\frac{n-\tau}{8}-\frac{\varepsilon}{4}n\\
 & \ge\frac{m}{4}-\frac{\Theta}{4}+\frac{n-(\Delta-\Theta+b)}{12}-\frac{\varepsilon}{4}n\\
 & =\frac{m}{6}+\frac{m+n-\Delta-2\Theta-b}{12}-\frac{\varepsilon}{4}n.
\end{align*}
By \eqref{eq:num_edges_min_two} and $\Theta\le\Delta\le n$,
we see that 
\[
m+n-\Delta-2\Theta-b\ge\left(b+2n-2\varepsilon n\right)+n-3n-b=-2\varepsilon n,
\]
from which our conclusion follows. Note that the inequality has the
least amount of slackness when all large vertices have in-surplus; see the
remark following (\ref{eq:num_edges_min_two}). This observation fits
well with the construction given in Section \ref{sec:Introduction}, where
there is a single large vertex having huge in-surplus.
This completes the proof of the minimum
outdegree $d=2$ case of Theorem \ref{thm:main}.

\section{Minimum outdegree three\label{sec:min-3}}
\label{sec:3}

The $d=3$ case of Theorem \ref{thm:main} is more complicated, and we
provide its proof in this section.  Suppose that $D=(V,E)$ is a given
directed graph of minimum outdegree at least 3 with $n$ vertices and $m$
edges, and $\varepsilon$ is a given positive real number. 
Our goal is to find a partition $V=V_{1}\cup V_{2}$ for which
$\min\{e(V_{1},V_{2}),e(V_{2},V_{1})\}\ge(\frac{1}{5}-\varepsilon)m$.
Throughout the proof we tacitly assume that the number of vertices $n$ is
large enough.

Suppose that $m\ge3200n$. Then by applying 
Proposition \ref{prop:second_moment} with $\varepsilon_{\ref{prop:second_moment}}=\frac{1}{20}$, 
we obtain a partition $V=V_{1}\cup V_{2}$ for which 
\[
\min\{e(V_{1,}V_{2}),e(V_{2},V_{1})\}\ge\frac{m}{4}-\frac{m}{20}=\frac{m}{5}.
\]
Hence it suffices to consider the case $m<3200n$. Since the minimum
degree of $D$ is at least three, we see that $3n\le m<3200n$.

Let $A$ be the set of \emph{large} vertices, which are defined as
vertices that have total degree at least $n^{3/4}$, and let $B=V\setminus A$.
Note that 
\[
|A|\cdot n^{3/4}\le2m<6400n,
\]
from which it follows that $|A|\le6400n^{1/4}\le\varepsilon n$ and
$e(A)\le(6400)^{2}n^{1/2}\le\frac{\varepsilon m}{2}$. By the same argument
which we used at the beginning of Section \ref{sec:2}, it suffices to
consider the case when $A$ induces no edges, and seek a partition
$V=V_{1}\cup V_{2}$ for which 
\[
\min\{e(V_{1,}V_{2}),e(V_{2},V_{1})\}\ge\frac{m}{5}-\frac{\varepsilon
m}{2}.
\]
Let $m_{A}=e(A,B)$ and $m_{B}=e(B)$, and note that $m=m_{A}+m_{B}$.

\subsection{Partition of large vertices}

This section follows essentially the same argument as Section
\ref{sec:2;gap}, but the proofs deviate in the next section.  As before, if
we define the gap $\Theta$ of a partition $A = A_1 \cup A_2$ to be
\[
\Theta=\big(e(A_{1},B)+e(B,A_{2})\big)-\big(e(B,A_{1})+e(A_{2},B)\big) ,
\]
we may take a partition which minimizes the magnitude of the gap, with $0
\leq \Theta \leq n$.  (The upper bound is provided by the greedy
partition.)

Since $\max_{v\in B}d(v)\le n^{3/4}\le\frac{\varepsilon^2}{16}m$, 
by Lemma \ref{lem:partition_secondmoment} with $p=\frac{1}{2}$
and $\varepsilon_{\ref{lem:partition_secondmoment}}=\frac{\varepsilon}{2}$,
there exists a partition $V=V_{1}\cup V_{2}$ such that
\[
\min\{e(V_{1},V_{2}),e(V_{2},V_{1})\}\ge\frac{m-\Theta}{4}-\frac{\varepsilon m}{2},
\]
and hence if $\Theta\le\frac{m}{5}$, then this partition already
gives a desired partition. Thus we may assume that 
\begin{equation}
\Theta>\frac{m}{5}\ge\frac{3n}{5}.\label{eq:gap_min_three}
\end{equation}
For the remainder of our proof, we will re-use the terms \emph{in-surplus},
\emph{out-surplus}, \emph{surplus}, \emph{forward vertex}, \emph{backward
vertex}, \emph{forward edge}, and \emph{backward edge} as originally
defined in Section \ref{sec:2;gap}.  Let $m_{A}^{f}$ and $m_{A}^{b}$ be the
numbers of forward and backward edges, respectively.

\subsection{Structure of large vertices}

Call a vertex $v\in A$ a \emph{huge vertex} if $s(v)\ge\Theta$.  Such
vertices exist for the same reason as in the $d=2$ case.
\begin{lem}
  \label{lem:large_structure}
  There exist at least one and at most four huge vertices, and the sum of
  the surpluses of the rest of the large vertices is at most $n-\Theta$.
\end{lem}
\begin{proof}
We have $\Theta>\frac{m}{5}$ by \eqref{eq:gap_min_three}, so it immediately
follows that there are at most four huge vertices.  By minimality of the
gap, if we switch the side of a forward vertex $v$, then we obtain a
partition whose gap is $\Theta-2s(v)$. Since we started with a partition
which minimized the absolute value of the gap, we must have $|\Theta -
2s(v)| \geq \Theta$.  Since the surplus $s(v)$ is always nonnegative, this
forces $\Theta - 2s(v) \leq -\Theta$, or equivalently, $s(v)\ge\Theta$.
Thus, all forward vertices have surplus at least $\Theta$, and are actually
huge. Moreover, since we chose the partition with $\Theta>0$, there are
more forward edges than backward edges, and we see that there exists at
least one forward vertex.

Now pick an arbitrary forward vertex $v$ of surplus $\Delta\ge\Theta$.
If we move $v$ to the other side, then the number of forward edges
becomes $m_{A}^{f}-\Delta$ and the number of backward edges becomes
$m_{A}^{b}+\Delta$. We have 
\[
m_{A}^{f}-\Delta\le m_{A}^{f}-\Theta = m_{A}^{b}
\quad \text{and} \quad
m_{A}^{f}=  m_{A}^{b}+\Theta \le m_{A}^{b}+\Delta.
\]
Since all forward vertices are huge, all the large vertices which are not huge are backward vertices.  If
the sum of the surpluses of these vertices is greater than $\Delta-\Theta$,
then since each of these vertices has surplus less than $\Theta$, by
choosing one such vertex at a time and switching its sides, we will
eventually reach a partition in which the number of forward edges is
greater than $m_{A}^{f}-\Delta+(\Delta-\Theta)=m_{A}^{f}-\Theta=m_{A}^{b}$
and less than $m_{A}^{f}-\Delta+\Theta \le m_{A}^{f}$. This contradicts
the minimality of the gap. Therefore the sum of the surpluses of the large
vertices that are not huge is at most $\Delta-\Theta\le n-\Theta$.
\end{proof}

To reduce the number of cases, our next step is to further restrict the
number of huge vertices.
\begin{lem}
  \label{lem:huge}
  The number of huge vertices is either one or three.
\end{lem}
\begin{proof}
By Lemma \ref{lem:large_structure}, we already know that the number of huge
vertices is between one and four inclusive.  Let $g$ be the sum of the
surpluses of the large vertices that are not huge.  By Lemma
\ref{lem:large_structure}, we have $g\le n-\Theta$.  Suppose that there are
two huge vertices $v_{1}$ and $v_{2}$, which have surpluses $\Delta_{1}$
and $\Delta_{2}$, respectively, where $\Delta_{1}\ge\Delta_{2}$.
Re-partition $A$ so that $v_{1}$ is the only forward vertex, and all the
other vertices are backward vertices.  Then the gap of this partition of
$A$ is $\Delta_{1}-(\Delta_{2}+g)$. However, $\Delta_{1}-(\Delta_{2}+g)\le n-(\Theta+0)$ 
and $(\Delta_{2}+g)-\Delta_{1}\le g\le n-\Theta$, and thus the magnitude 
of the gap of the new partition is at most $n-\Theta$, which is less than $\Theta$
by \eqref{eq:gap_min_three}. This is a contradiction.

Now suppose that there are four huge vertices $v_{1},v_{2},v_{3},v_{4}$,
which have surpluses $\Delta_{1}$, $\Delta_{2}$, $\Delta_{3}$,
and $\Delta_{4}$, respectively, where $\Delta_{1}\ge\Delta_{2}\ge\Delta_{3}\ge\Delta_{4}$.
Re-partition $A$ so that $v_{1}$ and $v_{3}$ are forward vertices,
and all other vertices are backward vertices. Then the gap of this
partition is $(\Delta_{1}+\Delta_{3})-(\Delta_{2}+\Delta_{4}+g)$. 
Since 
\[
(\Delta_{1}+\Delta_{3})-(\Delta_{2}+\Delta_{4}+g)\le\Delta_{1}-\Delta_{4}\le n-\Theta<\Theta
\]
and
\[
(\Delta_{2}+\Delta_{4}+g)-(\Delta_{1}+\Delta_{3})\le g\le n-\Theta<\Theta,
\]
this again gives a contradiction. Hence we either have one or three
huge vertices. 
\end{proof}

\subsection{One huge vertex}

This section completes the proof of our main theorem for $d=3$ when there
is only one huge vertex.  The final case with three huge vertices is
finished in the next section.  So, for this section, let $v_{0}$ be the
huge vertex and let $\Delta=s(v_{0})$. Since the sum of surpluses of all
the other large vertices is at most $n-\Theta$ (by Lemma
\ref{lem:large_structure}) and $n-\Theta < \Theta \leq \Delta$ (by
\eqref{eq:gap_min_three} and the greedy partition), we see that $v_{0}$ is
the unique forward vertex, and all the other large vertices are backward
vertices.  

Recall that the edges contributing to $\sum_{v\in A}\left(d(v)-s(v)\right)$
come in pairs of in-edges and out-edges. Call these the \emph{buffer
edges}, and let $2b=\sum_{v\in A}(d(v)-s(v))$. The observation above
implies that $m_{A}^{f}=\Delta+b$ and $m_{A}^{b}=\Delta-\Theta+b$.
Moreover, since the graph has minimum outdegree at least three, and there
are at least $b$ buffer edges directed into $B$, it also implies that the
number of edges in $D$ is at least 
\begin{equation}
  m
  \ge 
  b+3|B|
  \ge
  b+3n-3\varepsilon n.
  \label{eq:num_edges_min_three}
\end{equation}

In the $d=2$ case, at this point we applied Theorem
\ref{thm:random_bisection}, with a simple bound on the number of odd
components in $D[B]$.
There, it was sufficient to control the number of odd
components with only one vertex, which was easy because 1-vertex components
have extremely simple structure (they consist of a single vertex, inducing
no edges).  For the $d=3$ case, it turns out that we must also control the
number of 3-vertex odd components.  This would still be particularly easy
in the case of oriented graphs, where each pair of vertices spans at most
one edge, but an additional twist is required to handle the general case of
directed graphs, where edges can go in both directions between the same
pair of vertices.

Nevertheless, oriented graphs are still the extremal case, because there is
an additional way to gain from pairs of edges running in opposite
directions (which we call pairs of \emph{antiparallel}\/ edges).  We
strengthen Theorem \ref{thm:random_bisection} to take advantage of this
phenomenon.  Recall that 3-vertex tight components are undirected graphs
with the property that for every one of the 3 vertices, the remaining two
vertices form an edge, and the first vertex is either adjacent to both or
none of the other two.  A moment's inspection reveals that 3-vertex tight
components are undirected $K_3$'s.  This is particularly useful.

\begin{lem}
  \label{lem:3;random_bisection}
  Given any real constants $C,\varepsilon>0$, there exist $\gamma,n_{0}>0$
  for which the following holds. Let $D=(V,E)$ be a given directed graph
  with $n\geq n_{0}$ vertices and at most $Cn$ edges, and let $A\subset V$
  be a set of at most $\gamma n$ vertices which have already been
  partitioned into $A_{1}\cup A_{2}$.  Let $B=V\setminus A$, and suppose
  that every vertex in $B$ has degree at most $\gamma n$ (with respect to
  the full $D$). Let $\tau'$ be the number of \textbf{tight} components in
  the underlying undirected (simple) graph induced by $B$, \textbf{not
    counting the 3-vertex components which contain edges that lift to
    antiparallel pairs in $\boldsymbol{D}$}.  Then, there is a bipartition
    $V=V_{1}\cup V_{2}$ with $A_{1}\subset V_{1}$ and $A_{2}\subset V_{2}$,
    such that both 
  \begin{align*}
    e(V_{1},V_{2}) & \ge
    e(A_{1},A_{2})+\frac{e(A_{1},B)+e(B,A_{2})}{2}+\frac{e(B)}{4}+\frac{n-\tau'}{8}-\varepsilon n\\
    e(V_{2},V_{1}) & \ge
    e(A_{2},A_{1})+\frac{e(B,A_{1})+e(A_{2},B)}{2}+\frac{e(B)}{4}+\frac{n-\tau'}{8}-\varepsilon n\,.
  \end{align*}
\end{lem}

\noindent \textbf{Remark.} The only differences between this statement and
Theorem \ref{thm:random_bisection} are indicated in bold.  Importantly, the
bounds in the conclusion are the same.

\begin{proof}
  Let $\tau$ be the number of tight components in the underlying undirected
  (simple) graph induced by $B$, and let $\sigma$ be the number of 3-vertex
  tight components which contain an antiparallel pair when lifted to $D$.
  The first step of the proof of Theorem \ref{thm:random_bisection} was to
  apply Lemma \ref{lem:star_decompose} to partition $B = T_1 \cup \cdots
  \cup T_s \cup U$.  By the proof of that lemma, each 3-vertex tight
  component contributes exactly one star $T_i$, with its original $e_i$
  coming from two of the vertices, and the third vertex contributing a
  non-free vertex to $U$.  (Here, we use the fact that 3-vertex tight
  components are $K_3$'s, so that we are assured that the third vertex is
  always non-free.)  Using this structural fact again, observe that
  actually, no matter which edge of the $K_3$ we use for our $e_i$ to seed
  the $T_i$, the third vertex will always be non-free.  In particular, if
  the 3-vertex tight component contains an antiparallel pair, then we may
  select the corresponding edge in the underlying undirected (simple) graph
  as the $e_i$, and the total number of non-free edges will remain the same
  as before.  This is important because in the proof of Lemma
  \ref{lem:star_decompose}, it is essential that we start with a maximal
  matching, which secondarily maximizes the number of free vertices.

  Therefore, we may assume that $\sigma$ of the stars $T_i$ contain at
  least one edge which lifts to an antiparallel pair in $D$. This implies that
  the total number of edges in the stars $T_i$ is now $\sigma$ more than
  that in Theorem \ref{thm:random_bisection}. Continuing
  the proof along the lines of Theorem \ref{thm:random_bisection}, observe
  that the gain of $+\frac{1}{4}$ comes from the edges of $D$ which
  correspond to edges of the stars $T_i$.  Therefore, we may improve
  inequality \eqref{ineq:random_split_gain} to
  \begin{align*}
    \E{Y_{1}} 
    &\ge 
    e(A_{1},A_{2})+\frac{e(A_{1},B)
    +e(B,A_{2})}{2}+\frac{e(B)}{4}
    +\frac{1}{4} \left[
      \frac{(n-\gamma n)-(\tau+\varepsilon n)}{2}
      +
      \sigma
    \right] \\
    &\ge
    e(A_{1},A_{2})+\frac{e(A_{1},B)
    +e(B,A_{2})}{2}+\frac{e(B)}{4}
    +\frac{n-\tau'}{8}
    - \frac{(\varepsilon + \gamma) n}{8} ,
  \end{align*}
  because $\tau' = \tau - \sigma$.  At this point, we have reached the same
  formula as in the proof of Theorem \ref{thm:random_bisection}, except
  that $\tau$ has been fully replaced with $\tau'$.  Therefore, the rest of
  the proof completes in the same way as before.
\end{proof}

In order to use Lemma \ref{lem:3;random_bisection}, we must now control
$\tau'$.  We do this with a similar argument to what was used in Section
\ref{sec:2;finish}. 
\begin{lem}
  \label{lem:num_odd_component}
  The number of tight components in the underlying
  undirected (simple) graph induced by $B$, not counting 3-vertex
  components which contain edges that lift to antiparallel pairs in $D$,
  satisfies:
  \begin{displaymath}
    \tau'\le\frac{n+2(\Delta-\Theta+b)}{5}.
  \end{displaymath}
\end{lem}
\begin{proof}
Let $\tau_{1}$ be the number of isolated vertices, $\tau_{3}'$ be the
number of tight components of order three, not counting those which contain
antiparallel pairs, and $\tau_{5}$ be the number of odd components of order
at least five, each in the induced subgraph on $B$.  Note that
$\tau' \leq \tau_{1}+\tau_{3}'+\tau_{5}$.  By considering the number of
vertices, we obtain the inequality
\begin{equation}
\tau_{1}+3\tau_{3}'+5\tau_{5}\le n.\label{eq:odd_component_bound1}
\end{equation}

The vertices in $B$ must have outdegree at least three in the whole graph.
Each vertex has at most one edge incident to $v_{0}$, and there are at most
$\Delta-\Theta+b$ edges from $B$ to $A$ which are not incident to $v_{0}$.
Each isolated vertex in $B$ uses at least two edges out of the
$\Delta-\Theta+b$ edges. Similarly, since a 3-vertex component counted by
$\tau_3'$ contains at most 3 edges (it cannot have antiparallel pairs), in
order to obtain such a component, we must use at least three edges out of
the $\Delta-\Theta+b$ edges, per component. Thus we obtain the inequality
\begin{align*}
2\tau_{1}+3\tau_{3}' & \le\Delta-\Theta+b.
\end{align*}
By adding two times this inequality to \eqref{eq:odd_component_bound1},
we obtain
\[
  5\tau'+4\tau_{3}' \le 5\tau_{1}+9\tau_{3}'+5\tau_{5}\le n+2(\Delta-\Theta+b).
\]
Hence 
\[
\tau'\le\frac{n+2(\Delta-\Theta+b)}{5}.
\]

\end{proof}
Let $\gamma$ be the constant from Theorem \ref{thm:random_bisection},
where $C=3200$ and $\varepsilon_{\ref{thm:random_bisection}}=\frac{\varepsilon}{4}$.
Since $m\le3200n$, $|A| \leq 6400 \varepsilon n^{1/4} \le \gamma n$, 
and $\max_{v\in B}d(v)\le n^{3/4}\le\gamma n$,
by Theorem \ref{thm:random_bisection} and Lemma \ref{lem:num_odd_component},
we obtain a bipartition $V=V_{1}\cup V_{2}$ for which 
\begin{align*}
\min\{e(V_{1},V_{2}),e(V_{2},V_{1})\} &
\ge\frac{1}{2}\min\{m_{A}^{f},m_{A}^{b}\}+\frac{1}{4}m_{B}+\frac{n-\tau'}{8}-\frac{\varepsilon}{4}n.\\
 & \ge\frac{1}{4}(m-\Theta)+\frac{n}{8}-\frac{n+2(\Delta-\Theta+b)}{40}-\frac{\varepsilon}{4}n.\\
 & =\frac{1}{4}m-\frac{1}{5}\Theta+\frac{1}{10}n-\frac{1}{20}\Delta-\frac{1}{20}b-\frac{\varepsilon}{4}n.
\end{align*}
Thus it suffices to prove that the right hand side of above is at
least $\frac{m}{5}-\frac{\varepsilon n}{2}$, or equivalently that
\begin{align*}
 & \left(\frac{1}{4}m-\frac{1}{5}\Theta+\frac{1}{10}n-\frac{1}{20}\Delta-\frac{1}{20}b-\frac{\varepsilon}{4}n\right)-\left(\frac{m}{5}-\frac{\varepsilon}{2}n\right)\\
=\, & \frac{m}{20}-\frac{1}{5}\Theta+\frac{1}{10}n-\frac{1}{20}\Delta-\frac{1}{20}b+\frac{\varepsilon}{4}n
\end{align*}
is at least zero. Recall that by \eqref{eq:num_edges_min_three},
we have $m\ge b+3n-3\varepsilon n$. By substituting this bound on
$m$ in the equation above, we get
\begin{align*}
 & \frac{(b+3n-3\varepsilon n)}{20}-\frac{1}{5}\Theta+\frac{1}{10}n-\frac{1}{20}\Delta-\frac{1}{20}b+\frac{\varepsilon}{4}n\\
=\, & \frac{1}{4}n-\frac{1}{5}\Theta-\frac{1}{20}\Delta+\frac{\varepsilon}{10}n.
\end{align*}
Since $n\ge\Delta\ge\Theta$, the right hand side is indeed at least
zero, and this proves the theorem when there is one huge vertex.

\subsection{Three huge vertices}

Let $v_{1},v_{2},v_{3}$ be the three huge vertices, and let $\Delta_{1},\Delta_{2},\Delta_{3}$
be their respective surpluses so that $\Delta_{1}\ge\Delta_{2}\ge\Delta_{3}$.
Let $g$ be the sum of surpluses of the large vertices which are not
huge. By Lemma \ref{lem:large_structure}, we know that $g\le n-\Theta$.
Recall that the edges contributing to $\sum_{v\in A}\left(d(v)-s(v)\right)$
come in pairs of in-edges and out-edges. Call these the \emph{buffer
edges}, and let $2b=\sum_{v\in A}(d(v)-s(v))$.
Note that 
\begin{equation}
m=(\Delta_{1}+\Delta_{2}+\Delta_{3}+g+2b)+m_{B}.\label{eq:min_three_num_edges_A}
\end{equation}

We re-partition $A$ as follows. First place the three vertices $v_{1},v_{2},v_{3}$
so that $v_{1}$ is a forward vertex and $v_{2},v_{3}$ are backward
vertices. Depending on the range of parameters, we will choose where
to place the large vertices that are not huge. Let $m_{A}^{f}=\Delta_{1}+b+X$
and $m_{A}^{b}=\Delta_{2}+\Delta_{3}+b+Y$, where $X+Y=g$. We will
either use the partition that gives $(X,Y)=(g,0)$, or the partition
that gives $(X,Y)=(0,g)$.  Note that such partitions always exist.

If $v_{1}$ has positive out-surplus, then let $p=\frac{2}{5}$, and
if $v_{1}$ has positive in-surplus, then let $p=\frac{3}{5}$. By
Lemma \ref{lem:partition_secondmoment} with such choice of $p$ and
$\varepsilon_{\ref{lem:partition_secondmoment}}=\frac{\varepsilon}{2}$,
we obtain a bipartition of $V$ in which 
\begin{displaymath}
  e(V_{1},V_{2}) 
  \geq 
  (1-p)e(A_{1},B)+p\cdot e(B,A_{2})+p(1-p)e(B)-\frac{\varepsilon m}{2}.
\end{displaymath}
Also, note that $e(A_1,B) + e(B,A_2) = m_A^f$, and $\{p,1-p\} =
\big\{\frac{2}{5}, \frac{3}{5} \big\}$, so $(1-p)e(A_{1},B)+p\cdot
e(B,A_{2})$ has the form $\frac{3}{5} Z + \frac{2}{5} (m_A^f - Z)$ for
some $Z$.  By how we placed the vertex $v_{1}$, we always have $Z \geq
s(v_1)$, and therefore
\begin{displaymath}
(1-p)e(A_{1},B)+p\cdot e(B,A_{2}) 
\ge\frac{3}{5}s(v_{1})+\frac{2}{5}(m_{A}^{f}-s(v_{1}))
=\frac{3}{5}\Delta_{1}+\frac{2}{5}(b+X).
\end{displaymath}
Hence
\[
e(V_{1},V_{2})\ge\frac{3}{5}\Delta_{1}+\frac{2}{5}(b+X)+\frac{6}{25}m_{B}-\frac{\varepsilon m}{2},
\]
and for 
\[
m_{1,2}=\frac{3}{5}\Delta_{1}+\frac{2}{5}(b+X)+\frac{6}{25}m_{B},
\]
it suffices to prove that $m_{1,2}\ge\frac{m}{5}$.  For $e(V_2, V_1)$, we
simply use the observation that $\min\{p, 1-p\} = \frac{2}{5}$ together
with $e(A_2,B) + e(B,A_1) = m_A^b$, and therefore Lemma
\ref{lem:partition_secondmoment} gives
\begin{displaymath}
  e(V_2, V_1)
  \geq
  \frac{2}{5} m_A^b + \frac{6}{25} m_B - \frac{\varepsilon m}{2}.
\end{displaymath}
Hence for
\begin{align*}
m_{2,1} & =\frac{2}{5}m_{A}^{b}+\frac{6}{25}m_{B}=\frac{2}{5}(\Delta_{2}+\Delta_{3}+b+Y)+\frac{6}{25}m_{B},
\end{align*}
it suffices to prove that $m_{2,1}\ge\frac{m}{5}$.

Thus our goal is to show that $m_{1,2}-\frac{m}{5}$ and $m_{2,1}-\frac{m}{5}$
are both non-negative. By \eqref{eq:min_three_num_edges_A} asserting
$m=\Delta_{1}+\Delta_{2}+\Delta_{3}+g+2b+m_{B}$, we have 
\[
m_{1,2}-\frac{m}{5}=\frac{2\Delta_{1}-\Delta_{2}-\Delta_{3}}{5}+\frac{2X-g}{5}+\frac{1}{25}m_{B}
\]
and
\begin{align*}
m_{2,1}-\frac{m}{5} & =\frac{\Delta_{2}+\Delta_{3}-\Delta_{1}}{5}+\frac{2Y-g}{5}+\frac{1}{25}m_{B}.
\end{align*}
Two cases complete the rest of this section.

\medskip

\noindent \textbf{Case 1.} $2\Delta_{1}-\Delta_{2}-\Delta_{3}-g>0$.

\medskip

We partition $A$ so that $(X,Y)=(0,g)$. The condition in this case
immediately implies that $m_{1,2}-\frac{m}{5}\ge0$, and thus it suffices to
show that $m_{2,1}-\frac{m}{5}\ge0$. Note that since
$\Delta_{2},\Delta_{3}\ge\Theta$ and $\Delta_{1}\le n$, we have 
\begin{align*}
m_{2,1}-\frac{m}{5} & =\frac{\Delta_{2}+\Delta_{3}-\Delta_{1}}{5}+\frac{2Y-g}{5}+\frac{1}{25}m_{B}\\
 & \ge\frac{2\Theta-n}{5}+\frac{g}{5}+\frac{1}{25}m_{B}.
\end{align*}
By \eqref{eq:gap_min_three} asserting $\Theta>\frac{m}{5}\ge\frac{3n}{5}$,
we have $m_{2,1}-\frac{m}{5}\ge0$. Hence we obtain a desired partition.

\medskip

\noindent \textbf{Case 2.} $2\Delta_{1}-\Delta_{2}-\Delta_{3}-g\le0$.

\medskip

We partition $A$ so that $(X,Y)=(g,0)$. We have
\[
m_{1,2}-\frac{m}{5}=\frac{2\Delta_{1}-\Delta_{2}-\Delta_{3}}{5}+\frac{g}{5}+\frac{1}{25}m_{B},
\]
and this is non-negative since $\Delta_{1}\ge\Delta_{2}\ge\Delta_{3}$.  On
the other hand, 
\[
m_{2,1}-\frac{m}{5}=\frac{\Delta_{2}+\Delta_{3}-\Delta_{1}-g}{5}+\frac{1}{25}m_{B},
\]
and since $\Delta_{1}\le\frac{\Delta_{2}+\Delta_{3}+g}{2}$ by the condition
in this case, we have
\begin{align*}
m_{2,1}-\frac{m}{5} & \ge\frac{\Delta_{2}+\Delta_{3}}{10}-\frac{3g}{10}+\frac{1}{25}m_{B}.
\end{align*}
By \eqref{eq:gap_min_three} asserting $\Theta>\frac{3n}{5}$,
we have $\Delta_{2}+\Delta_{3}\ge2\Theta>\frac{6n}{5}$ and $3g<3(n-\Theta)<\frac{6n}{5}$.
Hence 
\[
m_{2,1}-\frac{m}{5}\ge0,
\]
and we obtain a desired partition. This concludes the proof.

\section{Concluding remarks}
\label{sec:concluding}

The structure of the proof for $d=2,3$ can be described as follows.  First,
we identify the vertices $A$ which have large total degree, and consider an
optimal partition of these vertices. We then further identify the ``huge''
vertices, which are vertices whose surplus is at least as large as the gap
of the partition. It turns out that there can only be a small number of
huge vertices. Finally, we partition the set $B=V\setminus A$ depending on
the structure of the huge vertices.  For this, we used two different
probabilistic approaches.  One was through the estimate on the number of
odd components (Theorem \ref{thm:random_bisection}), and another was
through making a random unbalanced partition of $B$ (Lemma
\ref{lem:partition_secondmoment}).  However, both methods turn out to be
too limited in strength to cover the cases $d\ge4$.

To see why we needed both probabilistic techniques, consider the
orientation of $K_{3,n-3}$ where all edges are oriented from the part of
size $n-3$ to the part of size 3.  This digraph essentially has minimum
outdegree 3, with only three vertices in violation, and a constant-size
addition would give it that property without affecting its asymptotic
partition performance.  So, we would expect there to be a partition for
which $\min\{e(V_{1},V_{2}),e(V_{2},V_{1})\} \ge \frac{m}{5}+o(m)$.  Note
that the set $A$ of large vertices would then be the three vertices of
degree $n-3$, and the remainder $B$ would be the $n-3$ vertices of degree
$3$.  If we try to use only Theorem \ref{thm:random_bisection}, the
resulting bipartition will nearly be a bisection (a bipartition into two
equal size parts), since that method distributes vertices into the two
sides with equal probability.  Yet if we only consider bisections  of this
graph, then in every bisection $V=V_{1}\cup V_{2}$, we have 
\[
  \min\{e(V_{1},V_{2}),e(V_{2},V_{1})\}
  \leq
  \left( \frac{1}{2} + o(1) \right) n
  =
  \left( \frac{1}{6} + o(1) \right) m ,
\]
which is already too small.

A different example shows that even an unbalanced straightforward random
partition of $B$ (Lemma \ref{lem:partition_secondmoment}) is insufficient.
Indeed, add a 3-out-regular graph inside the larger part of the bipartite
graph above, so that $m=6(n-3)$.  By merely taking a random partition of
$B$, and assuming that $V_{1}$ contains one vertex of degree $n-3$ and
$V_{2}$ contains two vertices of degree $n-3$, we obtain a partition for
which 
\begin{align*}
e(V_{1},V_{2}) & \approx 2np+p(1-p)\cdot3n\\
e(V_{2},V_{1}) & \approx n(1-p)+p(1-p)\cdot3n.
\end{align*}
In order to maximize $\min\{e(V_{1},V_{2}),e(V_{2},V_{1})\}$, we take
$p=\frac{1}{3}$, and obtain
$\min\{e(V_{1},V_{2}),e(V_{2},V_{1})\}\approx\frac{4}{3}n\approx\frac{2}{9}m$.
Even though this graph does not quite have minimum outdegree six, only
three vertices are deficient, and so if Conjecture \ref{conj:main} is true,
we expect there to be a bipartition for which
$\min\{e(V_{1},V_{2}),e(V_{2},V_{1})\}\ge\frac{5}{22}m+o(m)$.  Hence 
Lemma \ref{lem:partition_secondmoment} is also too weak on its own.

Therefore in order to proceed further under the same framework, we must
combine the two ideas. A naive combination will fail for the following
reason. Consider the orientation of $K_{5,n-5}$ where out of the 5 vertices
on one side, one vertex $v_{1}$ has outdegree $n-5$ and the other four
vertices $v_{2}, v_{3}, v_{4}, v_{5}$ have indegree $n-5$. The set $A$ of
large vertices is precisely $\{v_1, \ldots, v_5\}$. Suppose that the
optimal partition of $A$ has $v_{2} \in A_{1}$ and
$v_{1},v_{3},v_{4},v_{5}\in A_{2}$.  (It is possible to slightly modify the
graph to ensure that this is the unique optimal partition of $A$.) No
matter how we complete this partition into a partition of the whole vertex
set, we have 
\[
e(V_{2},V_{1})=n-5=\frac{m}{5}.
\]
Since the minimum outdegree is essentially 4, we need the factor
$\frac{3}{14}$ to prove Conjecture \ref{conj:main}, and thus we fall short.
Hence our example shows that in some cases we must start with a sub-optimal
partition of $A$. Indeed, we used this idea in our proof for the case
$d=3$, but in a brute force, ad-hoc manner. It would be interesting to find
a systematic way to combine all of these ideas to resolve the general case
$d\ge4$.

\medskip

\noindent \textbf{Acknowledgment}. We would like to thank the referee for
helpful comments.


\begin{thebibliography}{10}

\bibitem{ABKS}
N. Alon, B. Bollobas, M. Krivelevich and B. Sudakov, Maximum cuts and judicious partitions in graphs without short cycles, 
{\em J. Combinatorial Theory Ser. B 88} (2003), 329--346.

\bibitem{AKS}
N. Alon, M. Krivelevich and B. Sudakov, MaxCut in H-free graphs, {\em Combinatorics, Probability and Computing} 14 (2005), 629--647.

\bibitem{BoReTh93} B.~Bollob\'as, B. Reed, and A. Thomason An extremal
function for the achromatic number, \textbf{Graph Structure Theory},
Eds. N. Robertson and P. Seymour, American Mathematical Society, Providence,
Rhode Island (1993), 161--165.

\bibitem{BoSc99} B.~Bollob\'as and A.~Scott, Exact bounds for judicious
partitions of graphs, {\em Combinatorica} 19 (1999), 473--486.

\bibitem{BoSc00} B.~Bollob\'as and A.~Scott, Judicious partitions
of 3-uniform hypergraphs, {\em Eur. J. Combin.} 21 (2000), 289--300.

\bibitem{BoSc02a} B.~Bollob\'as and A.~Scott, Problems and results
on judicious partitions, {\em Random Structures and Algorithms}
21 (2002), 414--430.

\bibitem{BoSc10} B.~Bollob\'as and A.~Scott, Max $k$-cut and judicious
$k$-partitions, {\em Discrete Mathematics} 310 (2010), 2126--2139.

\bibitem{Edwards73} C.~Edwards, Some extremal properties of bipartite
subgraphs, {\em Canad. J. Math.} 25 (1973), 475--485.

\bibitem{Edwards75} C.~Edwards, An improved lower bound for the
number of edges in a largest bipartite subgraph, {\em Proc. 2nd
Czech Symp. Graph Theory, Prague} (1975), 167--181. 

\bibitem{ErGyKo97} P.~Erd\H{o}s, A.~Gy\'arf\'as, and Y.~Kohayakawa,
The size of the largest bipartite subgraphs, {\em Discrete Math.}
177 (1997), 267--271.

\bibitem{FrJe97} A.~Frieze and M.~Jerrum, Improved approximation
algorithms for MAX $k$-CUT and MAX BISECTION, {\em Algorithmica}
18 (1997), 61--77.

\bibitem{GoWi95} M.~Goemans and D.~Williamson, 0.878 approximation
algorithms for MAX CUT and MAX 2-SAT, {\em Proc. 26th ACM Symp.
Theory Comput.} (1994), 422--431; Updated as: Improved approximation
algorithms for maximum cut and satisfiability problems using semidefinite
programming, {\em J. ACM} 42 (1995), 1115--1145.

\bibitem{Ha11} J.~Haslegrave, 
  The Bollob\'as-Thomason conjecture for 3-uniform hypergraphs,
  {\em Combinatorica} 32 (2012), 451--471.

\bibitem{Hastad01} J.~H{\aa}stad, Some optimal inapproximability
results, {\em J. ACM} 48 (2001), 798--859.

\bibitem{JLR} S.~Janson, T.~\L{}uczak, and A. Ruci\'{n}ski, \textbf{Random
Graphs}, Wiley, New York (2000)

\bibitem{KuOs07} D.~K\"uhn and D.~Osthus, Maximizing several cuts
simultaneously, {\em Combinatorics, Probability and Computing}
16 (2007), 277--283.

\bibitem{LeLoSu} C.~Lee, P.~Loh, and B.~Sudakov, Bisections of
  graphs, {\em J. Comb. Theory, Ser. B}, to appear.

\bibitem{MaYaYu10} J.~Ma, P.~Yan, and X.~Yu, On several partition
problems of Bollob\'as and Scott, {\em J. Comb. Theory, Ser. B} 100
(2010), 631--649.

\bibitem{MaYu11} J.~Ma and X.~Yu, Partitioning 3-uniform hypergraphs,
{\em J. Comb. Theory, Ser. B} 102 (2012), 212--232.

\bibitem{PoTu82}S.~Poljak and Zs.~Tuza, Bipartite subgraphs of
triangle-free graphs, {\em SIAM J. Discrete Math.} 7 (1994), 307--313. 

\bibitem{Scott06} A.~Scott, Judicious partitions and related problems,
in {\em Surveys in combinatorics}, London Math. Soc. Lecture Note
Ser., 327, Cambridge Univ. Press, Cambridge (2005), 95--117.

\bibitem{TSSW00} L.~Trevisan, G.~Sorkin, M.~Sudan, and D.~Williamson,
Gadgets, approximation, and linear programming, {\em SIAM J. Comput.}
29 (2000), 2074--2097.\end{thebibliography}
\end{document}